\documentclass[11pt]{amsart}
\usepackage{amsfonts}
\usepackage{amsthm}
\usepackage{amsmath}
\usepackage{graphicx}
\usepackage{amscd,amssymb,amsthm}

\newcommand{\dt}{{\rm d}t}
\newcommand{\f}{\phi}
\newcommand{\ec}{\operatorname{erfc}}

\newcounter{minutes}
\divide\time by 60
\newcounter{hours}
\multiply\time by 60 \addtocounter{minutes}{-\time}

\title[Tur\'an type inequalities for Tricomi confluent hypergeometric functions]
{Tur\'an type inequalities for Tricomi confluent hypergeometric
functions}

\author[\'Arp\'ad Baricz]{\'Arp\'ad Baricz$\dag$}

\author[Mourad E.H. Ismail]{Mourad E.H. Ismail$\ddag$$\S$}



\newtheorem{theorem}{Theorem}
\newtheorem{remark}{Remark}

\begin{document}

\def\thefootnote{}
\footnotetext{ \texttt{File:~\jobname .tex,
          printed: \number\year-0\number\month-\number\day,
          \thehours.\ifnum\theminutes<10{0}\fi\theminutes}
} \makeatletter\def\thefootnote{\@arabic\c@footnote}\makeatother

\maketitle

\begin{center}
{\footnotesize $\dag$Department of Economics, Babe\c{s}-Bolyai
University,
Cluj-Napoca 400591, Romania\\ \email{bariczocsi@yahoo.com}\\
$\ddag$Department of Mathematics, University of Central Florida, Orlando, Florida 32816, United States\\
$\S$Department of Mathematics, King Saud University, Riyadh, Saudi Arabia\\
\email{mourad.eh.ismail@gmail.com}}
\end{center}

\begin{abstract}
Some sharp two-sided Tur\'an type inequalities for parabolic
cylinder functions and Tricomi confluent hypergeometric functions
are deduced. The proofs are based on integral representations for
quotients of parabolic cylinder functions and Tricomi confluent
hypergeometric functions, which arise in the study of the infinite
divisibility of the Fisher-Snedecor $F$ distribution. Moroever, some
complete monotonicity results are given concerning Tur\'an
determinants of Tricomi confluent hypergeometric functions. These
complement and improve some of the results of Ismail and Laforgia
\cite{il}.
\end{abstract}

\section{Introduction}

Since the publication in 1948 by Szeg\H o \cite{szego} of Tur\'an's
inequality for Legendre polynomials \cite{turan}, many researchers
have produced  analogous results for orthogonal polynomials and
special functions. In the last six decades it was shown by several
researchers that the most important special functions satisfy
Tur\'an type inequalities, see for example the most recent papers on this
topic written in the last five years \cite{alzer0}--\cite{berg},
\cite{il,karp0,karp,segura} and the references therein. Tur\'an type
inequalities seem to be evergreen in the theory of special
functions, nowadays they have an extensive literature and some of
the results have been applied successfully in problems which arise
in information theory, economic theory and biophysics. For more
details the interested reader is referred to the papers
\cite{bariczproc1,barnard,eli,sun}. Motivated by these applications,
recently the Tur\'an type inequalities have been investigated also
for hypergeometric and confluent  hypergeometric functions, as well as for the
generalized hypergeometric functions. See
\cite{baricz1,baricz3,barnard,karp} and the references therein for
more details.

In this paper we make a contribution to  the above mentioned
results by proving the corresponding sharp Tur\'an type inequalities
for parabolic cylinder functions and Tricomi confluent
hypergeometric functions. These results naturally complement  the
earlier results for Hermite polynomials, modified Bessel functions
of the second kind and Kummer confluent hypergeometric functions.

In Section 2 we consider the
parabolic cylinder functions and we prove a sharp Tur\'an type
inequality by using an integral formula from \cite{kelker}. In
Section 3 we establish Tur\'an type inequalities for the
Tricomi  $\psi$ (confluent hypergeometric) functions by using another  integral
representation formula from \cite{kelker}. The latter integral representation
 was the main tool
in the proof of the infinite divisibility of the Fisher-Snedecor $F$
distribution. Finally, in Section 4 we present a general result
concerning Tur\'an determinants whose  entries are
 functions having convenient integral representation.
These yield complete monotonicity results for Tur\'an
determinants of Tricomi confluent hypergeometric functions. The main
results of Sections 3 and 4 complement and improve the results of
Ismail and Laforgia \cite{il} concerning the Tricomi hypergeometric
function.

\section{Tur\'an type inequalities for parabolic cylinder functions}

The parabolic cylinder function or sometimes called as Weber
function $U(a,\cdot),$ denoted also as $D_{-a-1/2}$ following
Whittaker's notation, is a particular solution of Weber's
differential equation (see \cite[p. 687]{abra} or \cite[p.
116]{magnus2})
\begin{equation}\label{eqweber}
w''(x)-\left(a+\frac{x^2}{4}\right)w(x)=0
\end{equation}
and its value is represented explicitly as
$$U(a,x)=\frac{1}{2^{\eta}\sqrt{\pi}}\left[\cos(\eta\pi)\Gamma\left(\frac{1}{2}-
\eta\right)y_1(a,x)-\sqrt{2}\sin(\eta\pi)\Gamma(1-\eta)y_2(a,x)\right],$$
where
$$y_1(a,x)=\exp\left({-\frac{x^2}{4}}\right)\Phi\left(\frac{a}{2}+\frac{1}{4},\frac{1}{2},\frac{x^2}{2}\right)$$
and
$$y_2(a,x)=x\exp\left({-\frac{x^2}{4}}\right)\Phi\left(\frac{a}{2}+\frac{3}{4},\frac{3}{2},\frac{x^2}{2}\right)$$
are independent solutions of \eqref{eqweber}, $\eta=a/2+1/4$ and
$\Phi(a,c,\cdot)$ stands for the Kummer confluent hypergeometric
function, called also as confluent hypergeometric function of the
first kind.

Our first main result is Theorem \ref{th1} below.

\begin{theorem}\label{th1}
If $a>0$ and $x\in\mathbb{R},$ then the following Tur\'an type
inequalities are valid
\begin{equation}\label{turanpara}
0<D_{-a}^2(x)-D_{-a-1}(x)D_{-a+1}(x)\leq \mu_a,
\end{equation}
where
$$\mu_a=\frac{\pi}{2^a}\left[\frac{1}{\Gamma^2\left(\frac{a+1}{2}\right)}-
\frac{1}{\Gamma\left(\frac{a}{2}\right)\Gamma\left(\frac{a}{2}+1\right)}\right].$$
The left-hand side of \eqref{turanpara} is sharp as $|x|\to\infty$ and it is also valid when $a=0$ and $x>0.$ The equality is
attained on the right-hand side of \eqref{turanpara} when $x=0.$
\end{theorem}

\begin{proof}[\bf Proof]
Let us consider the Tur\'anian
\begin{align*}{}_D\Delta_a(x):&=D_{-a}^2(x)-D_{-a-1}(x)D_{-a+1}(x)\\&=
U^2\left(a-\frac{1}{2},x\right)-U\left(a-\frac{3}{2},x\right)U\left(a+\frac{1}{2},x\right),\end{align*}
which in view of the differential recurrence relations \cite[p. 688]{abra}
$$U'(a,x)+\frac{x}{2}U(a,x)+\left(a+\frac{1}{2}\right)U(a+1,x)=0$$
and
$$U'(a,x)-\frac{x}{2}U(a,x)+U(a-1,x)=0$$
can be rewritten as
$${}_D\Delta_a(x)=\left(1+\frac{x^2}{4a}\right)U^2\left(a-\frac{1}{2},x\right)-
\frac{1}{a}\left[U'\left(a-\frac{1}{2},x\right)\right]^2.$$ On the
other hand, since $U(a,x)$ satisfies the Weber differential equation
\eqref{eqweber}, we obtain
$$U''\left(a-\frac{1}{2},x\right)=\left(a-\frac{1}{2}+\frac{x^2}{4}\right)U\left(a-\frac{1}{2},x\right).$$
Moreover, the following integral representation formula \cite[p.
885]{kelker} is valid
\begin{equation}\label{quopara}
\frac{D_{-a-1}(\sqrt{z})}{\sqrt{z}D_{-a}(\sqrt{z})}=\frac{1}{\sqrt{2\pi}\Gamma(a+1)}\int_0^{\infty}
\frac{\left|D_{-a}(\mathrm{i}\sqrt{t})\right|^{-2}}{(z+t)\sqrt{t}}\dt,
\end{equation}
where $a>0$ and $|\arg z|<\pi.$ Now, by using \eqref{quopara} we get
\begin{align*}{}_D\Delta_a'(x)&=\frac{1}{a}U^2\left(a-\frac{1}{2},x\right)
\left[\frac{x}{2}+\frac{U'\left(a-\frac{1}{2},x\right)}{U\left(a-\frac{1}{2},x\right)}\right]\\
&=-U^2\left(a-\frac{1}{2},x\right)\left[\frac{U\left(a+\frac{1}{2},x\right)}{U\left(a-\frac{1}{2},x\right)}\right]\\
&=-D^2_{-a}(x)\left[\frac{D_{-a-1}(x)}{D_{-a}(x)}\right]\\
&=-D^2_{-a}(x)\int_0^{\infty}\frac{x}{x^2+t}\varphi_a(t)\dt,\end{align*}
where
$$\varphi_a(t)=\frac{\left|D_{-a}(\mathrm{i}\sqrt{t})\right|^{-2}}{\sqrt{2\pi}\Gamma(a+1)\sqrt{t}}.$$
Thus, the function $x\mapsto {}_D\Delta_a(x)$ is increasing on
$(-\infty,0]$ and decreasing on $[0,\infty),$ and consequently by
using \cite[p. 687]{abra}
$$U(a,0)=\frac{\sqrt{\pi}}{2^{\frac{a}{2}+\frac{1}{4}}\Gamma\left(\frac{a}{2}+\frac{3}{4}\right)}$$
we have ${}_D\Delta_a(x)\leq {}_D\Delta_a(0)=\mu_a$ for all
$x\in\mathbb{R}$ and $a>0.$ This proves the inequality on the right-hand side of
\eqref{turanpara}. Now, for the inequality on the  left-hand side of \eqref{turanpara}
recall that for $|\arg z|<\pi/2$ we have \cite[p. 689]{abra}
$$\lim_{|z|\to\infty}\frac{U(a,z)}{e^{-\frac{z^2}{4}}z^{-a-\frac{1}{2}}}=1,$$
and then $\lim\limits_{|x|\to\infty}{}_D\Delta_a(x)=0,$ which
completes the proof when $a>0.$

Finally, recall that \cite[p. 692]{abra}
$$U\left(-\frac{1}{2},x\right)=e^{-\frac{x^2}{4}},\ \ U\left(-\frac{3}{2},x\right)=xe^{-\frac{x^2}{4}}$$
and $$U\left(\frac{1}{2},x\right)=\sqrt{\frac{\pi}{2}}e^{\frac{x^2}{4}}\ec\left(\frac{x}{\sqrt{2}}\right),$$
where $\ec,$ defined by
$$\ec(x)=\frac{2}{\sqrt{\pi}}\int_x^{\infty}e^{-t^2}\dt,$$
denotes the complementary error function. Thus, we obtain for all $x>0$ that
$${}_D\Delta_{0}(x)=D_0^2(x)-D_{-1}(x)D_1(x)=e^{-\frac{x^2}{2}}-x\int_x^{\infty}e^{-\frac{t^2}{2}}\dt>0,$$
which is exactly the upper bound inequality of Gordon \cite{gordon} for the Mills ratio of the standard normal distribution. For more details see also \cite[p. 1363]{mills}. This completes the proof.
\end{proof}

\begin{remark}
{\em We mention that recently, by using a different approach, Segura \cite[Theorem 11]{segura2} proved the inequalities
$$1<\frac{D_{-a}^2(x)}{D_{-a-1}(x)D_{-a+1}(x)}<\sqrt{\frac{a+1}{a-1}},$$
where $a>0$ and $x\in\mathbb{R}$ on the left-hand side, and $a>1$ and $x\in\mathbb{R}$ on the right-hand side. We also note here that the inequality on the left-hand side of \eqref{turanpara} complements the Tur\'an type inequality for Hermite
polynomials (see \cite{lak,madhava,szego})
$$H_n^2(x)-H_{n-1}(x)H_{n+1}(x)\geq0,$$
which is valid for all real $x$ and $n\in\{1,2,\dots\}.$ Indeed, if
$n$ is a non-negative integer, then by using the relation \cite[p.
780]{abra}
$$D_n(x)=2^{-\frac{n}{2}}e^{-\frac{x^2}{4}}H_n\left(\frac{x}{\sqrt{2}}\right),$$
the above Tur\'an type inequality becomes
$$D_n^2(x)-D_{n-1}(x)D_{n+1}(x)\geq0,$$
where $n\in\{1,2,\dots\}$ and $x\in\mathbb{R}.$ The left-hand side
of \eqref{turanpara} complements this inequality.

Now, let us consider the cases when $a\in\{-3/2,-1/2\}.$ Observe that in view of \cite[p. 692]{abra}
$$U(0,x)=\sqrt{\frac{x}{2\pi}}K_{\frac{1}{4}}\left(z\right),\ \ U(-1,x)=\frac{x\sqrt{x}}{2\sqrt{2\pi}}\left[K_{\frac{1}{4}}\left(z\right)+K_{\frac{3}{4}}\left(z\right)\right],$$
$$U(-2,x)=\frac{x^2\sqrt{x}}{4\sqrt{2\pi}}\left[2K_{\frac{1}{4}}\left(z\right)+3K_{\frac{3}{4}}\left(z\right)-K_{\frac{5}{4}}\left(z\right)\right]$$
and
$$U(-3,x)=\frac{x^3\sqrt{x}}{8\sqrt{2\pi}}\left[5K_{\frac{1}{4}}\left(z\right)+9K_{\frac{3}{4}}\left(z\right)-5K_{\frac{5}{4}}\left(z\right)-K_{\frac{7}{4}}\left(z\right)\right]$$
we have that
\begin{align*}{}_D\Delta_{-\frac{3}{2}}(x)&=U^2(-2,x)-U(-3,x)U(-1,x)\\
&=\frac{x^5}{32\pi}\left[K_{\frac{3}{4}}\left(z\right)\left(K_{\frac{7}{4}}\left(z\right)-K_{\frac{5}{4}}\left(z\right)\right)+
K_{\frac{5}{4}}^2\left(z\right)-K_{\frac{1}{4}}^2\left(z\right)\right.\\&\left.+K_{\frac{1}{4}}\left(z\right)\left(K_{\frac{5}{4}}\left(z\right)+K_{\frac{7}{4}}\left(z\right)-2K_{\frac{3}{4}}\left(z\right)\right)\right],\end{align*}
and
\begin{align*}{}_D\Delta_{-\frac{1}{2}}(x)&=U^2(-1,x)-U(-2,x)U(0,x)\\
&=\frac{x^3}{8\pi}\left[K_{\frac{1}{4}}\left(z\right)\left(K_{\frac{5}{4}}\left(z\right)-K_{\frac{3}{4}}\left(z\right)\right)+
K_{\frac{3}{4}}^2\left(z\right)-K_{\frac{1}{4}}^2\left(z\right)\right],\end{align*}
where $z=x^2/4$ and $K_a$ stands for the modified Bessel function of the second kind.
By using the known fact (see \cite{laforgia}) that $a\mapsto K_a(x)$ is increasing on $(0,\infty)$ for each fixed $x>0,$ it follows that ${}_D\Delta_{-\frac{3}{2}}(x)>0$ and ${}_D\Delta_{-\frac{1}{2}}(x)>0$ for all $x>0.$ Numerical experiments and the above results for $a\in\{-3/2,-1/2,0\}$ suggest that the left-hand side of the inequality \eqref{turanpara} is also valid for all $a\leq0$ and $x>0,$ however, we were unable to prove this.}
\end{remark}

\section{Tur\'an type inequalities for Tricomi $\psi$ function}
\setcounter{equation}{0}

The Tricomi's confluent hypergeometric function, also called
confluent hypergeometric function of the second kind,
$\psi(a,c,\cdot)$ is a particular solution of the confluent
hypergeometric  differential equation (see \cite[p. 504]{abra} or
\cite[p. 248]{magnus})
\begin{equation}\label{eqkummer}
xw''(x)+(c-x)w'(x)-aw(x)=0
\end{equation}
and its value is defined in terms of the Kummer confluent
hypergeometric function as
$$\psi(a,c,x)=\frac{\Gamma(1-c)}{\Gamma(a-c+1)}\Phi(a,c,x)+
\frac{\Gamma(c-1)}{\Gamma(a)}x^{1-c}\Phi(a-c+1,2-c,x).$$

Now, recall the following Tur\'an type inequalities, which hold for
all $a>1$ and $x>0$
\begin{equation}\label{tmb}
\frac{1}{1-a}K_{a}^2(x)<K_{a}^2(x)-K_{a-1}(x)K_{a+1}(x)<0.
\end{equation}
Moreover, the right-hand side of \eqref{tmb} holds true for all
$a\in\mathbb{R}.$ These inequalities are sharp in the sense that the
constants $1/(1-a)$ and $0$ are best possible.

For the sake of completeness it should be mentioned that the
right-hand side of \eqref{tmb} was first proved independently by
Ismail and Muldoon \cite{muldoon} and van Haeringen \cite{Har}, and
rediscovered later by Laforgia and Natalini \cite{natalini}. Note
that in \cite{muldoon} the authors actually proved that for all
fixed $x>0$ and $b>0,$ the function $a\mapsto K_{a+b}(x)/K_{a}(x)$
is increasing on $\mathbb{R}.$ Another proof of the right-hand side
of \eqref{tmb}, with proper credit, is in \cite{bariczstudia}.
Recently, Baricz \cite{bams} and Segura \cite{segura}, proved the
two sided inequality in \eqref{tmb} by using different approaches.
 See also \cite{bepe} for more
details on \eqref{tmb}. We also note here that the left-hand side of
\eqref{tmb} provides actually an upper bound for the effective
variance of the generalized Gaussian distribution. More precisely,
in \cite{lacis} the authors used (without proof) the inequality
$0<v_{\rm{eff}}<1/(\mu-1)$ for $\mu=a+4,$ where
$$v_{\rm{eff}}=\frac{K_{\mu-1}(x)K_{\mu+1}(x)}{K_{\mu}^2(x)}-1$$
is the effective variance of the generalized Gaussian distribution.

Observe that by using the relation \cite[p. 510]{abra}
$$K_{a}(x)=\sqrt{\pi}(2x)^{a}e^{-x}\psi\left(a+\frac{1}{2},2a+1,2x\right)$$
the Tur\'an type inequality \eqref{tmb} for $a>1$ and $x>0$ can be
rewritten as
\begin{equation}\label{turanhankel}\frac{1}{1-a}\psi^2\left(a+\frac{1}{2},2a+1,x\right)<\Delta_a(x)<0,\end{equation}
where
$$\Delta_a(x):=\psi^2\left(a+\frac{1}{2},2a+1,x\right)-
\psi\left(a-\frac{1}{2},2a-1,x\right)\psi\left(a+\frac{3}{2},2a+3,x\right).$$

The next result complements  the above inequality.

\begin{theorem}
If $a>0>c$ and $x>0,$ then the following sharp Tur\'an type
inequalities are valid
\begin{equation}\label{turantrico}
\frac{1}{c}\psi^2(a,c,x)<\psi^2(a,c,x)-\psi(a-1,c-1,x)\psi(a+1,c+1,x)<0.
\end{equation}
Moreover, the right-hand side of \eqref{turantrico} holds true for
all $a>0,$ $c<1$ and $x>0.$ These inequalities are sharp in the sense that the
constants $1/c$ and $0$ are best possible.
\end{theorem}

\begin{proof}[\bf Proof]
First consider the expression
\begin{align*}
{}_{\psi}\Delta_{a,c}(x):=\psi^2(a,c,x)-\psi(a-1,c-1,x)\psi(a+1,c+1,x),
\end{align*}
which in view of the relations \cite[p. 507]{abra}
$$\psi'(a,c,x)=-a\psi(a+1,c+1,x),$$
$$\psi(a-1,c-1,x)=(1-c+x)\psi(a,c,x)-x\psi'(a,c,x)$$
can be rewritten as
$${}_{\psi}\Delta_{a,c}(x)=\psi^2(a,c,x)+\frac{1}{a}
(1-c+x)\psi(a,c,x)\psi'(a,c,x)-\frac{x}{a}\left[\psi'(a,c,x)\right]^2.$$
On the other hand, because $\psi(a,c,x)$ satisfies the confluent
differential equation \eqref{eqkummer}, we obtain
\begin{equation}\label{fo}\left[\frac{x\psi'(a,c,x)}{\psi(a,c,x)}\right]'=
(1+x-c)\frac{\psi'(a,c,x)}{\psi(a,c,x)}+a-x\left[\frac{\psi'(a,c,x)}{\psi(a,c,x)}\right]^2\end{equation}
and conclude that
$$\frac{{}_{\psi}\Delta_{a,c}(x)}{\psi^2(a,c,x)}=\frac{1}{a}\left[\frac{x\psi'(a,c,x)}{\psi(a,c,x)}\right]'
=-\left[\frac{x\psi(a+1,c+1,x)}{\psi(a,c,x)}\right]'.$$ For $|\arg
z|<\pi,$ $a>0$ and $c<1$ the integral representation \cite[p.
885]{kelker}
\begin{equation}\label{quotrico}
\frac{\psi(a+1,c+1,z)}{\psi(a,c,z)}=\int_0^{\infty}\frac{t^{-c}e^{-t}
\left|\psi(a,c,te^{\mathrm{i}\pi})\right|^{-2}}
{(z+t)\Gamma(a+1)\Gamma(a-c+1)}\dt
\end{equation}
is valid. By using the notation
$$\varphi_{a,c}(t):=\frac{t^{-c}e^{-t}\left|\psi(a,c,te^{\mathrm{i}\pi})\right|^{-2}}
{\Gamma(a+1)\Gamma(a-c+1)},$$ \eqref{quotrico} implies that
$$\frac{{}_{\psi}\Delta_{a,c}(x)}{\psi^2(a,c,x)}=-\left[\int_0^{\infty}\frac{x}{x+t}\varphi_{a,c}(t)\dt\right]'
=-\int_0^{\infty}\frac{t\varphi_{a,c}(t)\dt}{(x+t)^2}<0$$ and
$$\left[\frac{{}_{\psi}\Delta_{a,c}(x)}{\psi^2(a,c,x)}\right]'
=-\left[\int_0^{\infty}\frac{t\varphi_{a,c}(t)\dt}{(x+t)^2}\right]'
=\int_0^{\infty}\frac{2t\varphi_{a,c}(t)\dt}{(x+t)^3}>0$$ for all
$a>0,$ $c<1$ and $x>0.$ Thus, the function $x\mapsto
{{}_{\psi}\Delta_{a,c}(x)}/{\psi^2(a,c,x)}$ maps $(0,\infty)$ into
$(-\infty,0)$ and it is strictly increasing. Hence, we obtain for all
$a>0,$ $c<1$ and $x>0$
$$\alpha_{a,c}:=\lim_{x\to 0}\frac{{}_{\psi}\Delta_{a,c}(x)}{\psi^2(a,c,x)}
<\frac{{}_{\psi}\Delta_{a,c}(x)}{\psi^2(a,c,x)}
<\lim_{x\to\infty}\frac{{}_{\psi}\Delta_{a,c}(x)}{\psi^2(a,c,x)}=:\beta_{a,c}.$$
The asymptotic expansion \cite[p. 508]{abra}
\begin{equation}\label{asymp1}\psi(a,c,x)\sim
x^{-a}\left(1+a(c-a-1)\frac{1}{x}+\frac{1}{2}a(a+1)(a+1-c)(a+2-c)\frac{1}{x^2}+\dots\right),\end{equation}
which is valid for large real $x$ and fixed $a$ and $c,$ implies
that $\beta_{a,c}=0.$ Similarly, by using the asymptotic expansion
\cite[p. 508]{abra} \begin{equation}\label{asymp2}\psi(a,c,x)\sim
\frac{\Gamma(1-c)}{\Gamma(1+a-c)},\end{equation} where $c<1$ and
$a>0$ are fixed and $x\to 0,$ we see that $\alpha_{a,c}=1/c$ for
$c<0.$ It is clear by construction that the constants $\alpha_{a,c}$
and $\beta_{a,c}$ are best possible.
\end{proof}

\begin{remark}
{\em Observe that by using \cite[p. 505]{abra}
\begin{equation}\label{whittaker}W_{\kappa,\mu}(x)=
\exp\left(-\frac{x}{2}\right)x^{\mu+\frac{1}{2}}\psi\left(\mu-\kappa+\frac{1}{2},1+2\mu,x\right),\end{equation}
for $x>0$ the inequality \eqref{turantrico} can be rewritten in
terms of Whittaker functions $W_{\kappa,\mu}$ as follows
\begin{equation}\label{turanwhitta}\frac{1}{1+2\mu}W_{\kappa,\mu}^2(x)<W_{\kappa,\mu}^2(x)-
W_{\kappa+\frac{1}{2},\mu-\frac{1}{2}}(x)W_{\kappa-\frac{1}{2},\mu+\frac{1}{2}}(x)<0,\end{equation}
where the left-hand side holds for $0>\mu+1/2>\kappa,$ while the
right-hand side is valid for $1/2>\mu+1/2>\kappa.$}
\end{remark}

\begin{remark}\label{remholder}
{\em We note that by using \eqref{quotrico} directly we can prove a
weaker Tur\'an type inequality than the right-hand side of
\eqref{turantrico}. More precisely, because of \eqref{quotrico} (see
also \cite[p. 889]{kelker}) the function
$$x\mapsto G(x):=-\frac{\psi'(a,c,x)}{\psi(a,c,x)}=\frac{a\psi(a+1,c+1,x)}{\psi(a,c,x)}$$
is a Stieltjes transform, and consequently it is strictly completely
monotonic, i.e. for all $a>0,$ $c<1$ and $x>0$ we have
$(-1)^nG^{(n)}(x)>0,$ which in particular implies that the function
$x\mapsto{\psi'(a,c,x)}/{\psi(a,c,x)}$ is increasing and then the
Laguerre type inequality
$${\psi''(a,c,x)}{\psi(a,c,x)}-\left[{\psi'(a,c,x)}\right]^2< 0$$
is valid. This is equivalent to
$$a{\psi^2(a+1,c+1,x)}-(a+1){\psi(a,c,x)}{\psi(a+2,c+2,x)}<0$$
or to
\begin{equation}\label{turantrico2}{\psi^2(a,c,x)}-
{\psi(a-1,c-1,x)}{\psi(a+1,c+1,x)}<\frac{1}{a}{\psi^2(a,c,x)},\end{equation}
where $a>1,$ $c<2$ and $x>0.$

Moreover, by using \cite[p. 505]{abra}
\begin{equation}\label{intlap}\psi(a,c,x)=\frac{1}{\Gamma(a)}\int_0^{\infty}e^{-xt}t^{a-1}(1+t)^{c-a-1}\dt,\end{equation}
the restriction $c<2$ in the inequality \eqref{turantrico2} can be
removed. More precisely, the H\"older-Rogers inequality for
integrals implies for all $a_1,a_2>0,$ $c_1,c_2\in\mathbb{R},$ $x>0$
and $\alpha\in[0,1]$
\begin{align*}
\Gamma&(\alpha a_1+(1-\alpha)a_2)\psi(\alpha
a_1+(1-\alpha)a_2,\alpha
c_1+(1-\alpha)c_2,x)\\&=\int_0^{\infty}e^{-xt}t^{\alpha
a_1+(1-\alpha)a_2-1}(1+t)^{\alpha c_1+(1-\alpha)c_2-(\alpha
a_1+(1-\alpha)a_2)-1}\dt\\
&=\int_0^{\infty}\left(e^{-xt}t^{a_1-1}(1+t)^{c_1-a_1-1}\right)^{\alpha}
\left(e^{-xt}t^{a_2-1}(1+t)^{c_2-a_2-1}\right)^{1-\alpha}\dt\\
&< \left[\int_0^{\infty}e^{-xt}t^{a_1-1}(1+t)^{c_1-a_1-1}\dt
\right]^{\alpha}\left[\int_0^{\infty}e^{-xt}t^{a_2-1}(1+t)^{c_2-a_2-1}\dt\right]^{1-\alpha}\\
&=\left[\Gamma(a_1)\psi(a_1,c_1,x)\right]^{\alpha}\left[\Gamma(a_2)\psi(a_2,c_2,x)\right]^{1-\alpha}
\end{align*}
and then the two-variable function $(a,c)\mapsto
\Gamma(a)\psi(a,c,x)$ is strictly logarithmically convex for each
$a,x>0$ and $c\in\mathbb{R}.$ Now, observe that the above inequality
in particular for $\alpha=1/2,$ $a_1=a-1,$ $a_2=a+1,$ $c_1=c-1$ and
$c_2=c+1$ reduces to \eqref{turantrico2}.}
\end{remark}

\begin{remark}\label{remarkholder}
{\em Observe that by using the above idea mutatis mutandis the
function $a\mapsto \Gamma(a)\psi(a,c,x)$ is also strictly
logarithmically convex for $a,x>0$ and $c\in\mathbb{R},$ and
consequently the Tur\'an type inequality
\begin{equation}\label{eqrem1}{\psi^2(a,c,x)}-
{\psi(a-1,c,x)}{\psi(a+1,c,x)}<\frac{1}{a}{\psi^2(a,c,x)}\end{equation}
is valid for all $a>1$ and $c,x\in\mathbb{R}.$ Taking into account
the relation \cite[p. 510]{abra}
$$D_{-a}(x)=2^{-\frac{a}{2}}\exp\left(-\frac{x^2}{4}\right)
\psi\left(\frac{a}{2},\frac{1}{2},\frac{x^2}{2}\right)$$
the above inequality in particular reduces to
\begin{equation}\label{eqrem2}D_{-2a}^2(x)-D_{-2a-2}(x)D_{-2a+2}(x)<\frac{1}{a}D_{-2a}^2(x),\end{equation}
which resembles to \eqref{turanpara}. However, in the above Tur\'an
type inequalities the constant $1/a$ is not best possible, as we
shall see below.}
\end{remark}

The next result is similar to \eqref{turantrico}.

\begin{theorem}
If $a>1>c$ and $x>0,$ then the next sharp Tur\'an type inequality is
valid
\begin{equation}\label{turantricom}\frac{1}{1+a-c}{\psi^2(a,c,x)}>
{\psi^2(a,c,x)}-{\psi(a-1,c,x)}{\psi(a+1,c,x)}>0.\end{equation}
Moreover, the right-hand side of \eqref{turantricom} holds true for
all $a>0,$ $c<1$ and $x>0.$ These inequalities are sharp in the sense that the
constants $1/(1+a-c)$ and $0$ are best possible.
\end{theorem}

\begin{proof}[\bf Proof] The proof is very similar to the proof of
\eqref{turantrico}, so we only sketch the proof. By using the
recurrence relations \cite[p. 507]{abra}
$$\psi(a-1,c,x)=(a-c+x)\psi(a,c,x)-x\psi'(a,c,x)$$
and
$$a(1+a-c)\psi(a+1,c,x)=a\psi(a,c,x)+x\psi'(a,c,x),$$
the expression
$${}_{\psi}\Delta_a(x):={\psi^2(a,c,x)}-{\psi(a-1,c,x)}{\psi(a+1,c,x)}$$
can be rewritten as
$${}_{\psi}\Delta_a(x)=\frac{1-x}{1+a-c}\psi^2(a,c,x)-
\frac{x(x-c)}{a(1+a-c)}\psi(a,c,x)\psi'(a,c,x)+
\frac{x^2}{a(1+a-c)}\left[\psi'(a,c,x)\right]^2.$$
In view of \eqref{fo} and \eqref{quotrico} this implies that
$$\frac{(1+a-c){}_{\psi}\Delta_a(x)}{\psi^2(a,c,x)}=1+\frac{x}{a}\frac{\psi'(a,c,x)}{\psi(a,c,x)}-
\frac{x}{a}\left[\frac{x\psi'(a,c,x)}{\psi(a,c,x)}\right]'
=1-\int_0^{\infty}\frac{x^2\varphi_{a,c}(t)\dt}{(x+t)^2}$$
and then
$$\left[\frac{(1+a-c){}_{\psi}\Delta_a(x)}{\psi^2(a,c,x)}\right]'
=-\int_0^{\infty}\frac{2xt\varphi_{a,c}(t)\dt}{(x+t)^3}<0$$
for all $a>0,$ $c<1$ and $x>0.$ Consequently the function $x\mapsto
{{}_{\psi}\Delta_{a}(x)}/{\psi^2(a,c,x)}$ is strictly decreasing on
$(0,\infty),$ which implies that for all $a>1,$ $c<1$ and $x>0$ we
have
$$\alpha_{a}:=\lim_{x\to 0}\frac{{}_{\psi}\Delta_{a}(x)}{\psi^2(a,c,x)}>\frac{{}_{\psi}\Delta_{a}(x)}{\psi^2(a,c,x)}
>\lim_{x\to\infty}\frac{{}_{\psi}\Delta_{a}(x)}{\psi^2(a,c,x)}=:\beta_{a},$$
where $\alpha_a=1/(1+a-c)$ and $\beta_a=0$ in view of the asymptotic
expansions \eqref{asymp1} and \eqref{asymp2}. Moreover, the
right-hand side of the above inequality is valid for all $a>0,$
$c<1$ and $x>0.$
\end{proof}

\begin{remark}
{\em Observe that the left-hand side of the Tur\'an type inequality
\eqref{turantricom} improves the inequality \eqref{eqrem1}, and the
constant $1/(1+a-c)$ cannot be improved. Moreover, we note that in
particular the inequality \eqref{turantricom} becomes
$$0<D_{-2a}^2(x)-D_{-2a-2}(x)D_{-2a+2}(x)<\frac{1}{a+1/2}D_{-2a}^2(x),$$
where $a>0$ and $x>0$ on the left-hand side, and $a>1$ and $x>0$ on
the right-hand side. Clearly, the right-hand side of this inequality
is an improvement over \eqref{eqrem2} and the constant $1/(a+1/2)$
is optimum.

Finally, observe that by using \eqref{whittaker} for $x>0$ the
inequality \eqref{turantricom} can be rewritten as follows
$$\frac{1}{-\mu-\kappa+1/2}W_{\kappa,\mu}^2(x)>W_{\kappa,\mu}^2(x)-
W_{\kappa-1,\mu}(x)W_{\kappa+1,\mu}(x)>0,$$ where the left-hand side
is valid for $-1/2>\mu-1/2>\kappa,$ while the right-hand side is
valid for $1/2>\mu+1/2>\kappa.$}
\end{remark}

The following result is a companion of \eqref{turantrico} and
\eqref{turantricom}.

\begin{theorem}
The function $c\mapsto \psi(a,c,x)$ is strictly logarithmically
convex on $\mathbb{R}$ for all $a,x>0$ fixed, and the following
sharp Tur\'an type inequality
\begin{equation}\label{turantrico3}\frac{a}{c(1+a-c)}{\psi^2(a,c,x)}
<{\psi^2(a,c,x)}-{\psi(a,c-1,x)}{\psi(a,c+1,x)}<0\end{equation} is
valid for all $a>0>c$ and $x>0.$ Furthermore, the right-hand side of
\eqref{turantrico3} holds for all $a,x>0$ and $c\in\mathbb{R}.$ These inequalities are sharp in the sense that the
constants $a(c(1+a-c))^{-1}$ and $0$ are best possible. In
addition, the sharp inequality
\begin{equation}\label{turantrico4}\frac{1}{2-c}{\psi^2(a,c,x)}
<{\psi^2(a,c,x)}-{\psi(a,c-1,x)}{\psi(a,c+1,x)}<0\end{equation}
is also valid for all $x>0$ and $a>c-1>1$ in the case of the
left-hand side, and $a>c-1>0$ in the case of the right-hand side.
These inequalities are sharp in the sense that the
constants $1/(2-c)$ and $0$ are best possible.
\end{theorem}

\begin{proof}[\bf Proof]
The proof of the strict logarithmic convexity of $c\mapsto
\psi(a,c,x)$ goes along the lines outlined in Remark
\ref{remholder}, so we shall omit the details. The right-hand side
of \eqref{turantrico3} follows from this strict logarithmic
convexity property, however, we give here an alternative proof,
which is similar to the proof of \eqref{turantrico}. For this
consider the Tur\'anian
\begin{align*}
{}_{\psi}\Delta_c(x):=\psi^2(a,c,x)-\psi(a,c-1,x)\psi(a,c+1,x),
\end{align*}
which by using the relations \cite[p. 507]{abra}
$$(1+a-c)\psi(a,c-1,x)=(1-c)\psi(a,c,x)-x\psi'(a,c,x)$$
and
$$\psi(a,c+1,x)=\psi(a,c,x)-\psi'(a,c,x),$$
can be rewritten as
$${}_{\psi}\Delta_c(x)=\frac{a}{1+a-c}\psi^2(a,c,x)+\frac{1+x-c}{1+a-c}
\psi(a,c,x)\psi'(a,c,x)-\frac{x}{1+a-c}\left[\psi'(a,c,x)\right]^2.$$
Consequently we have
$$\frac{(1+a-c){}_{\psi}\Delta_c(x)}{\psi^2(a,c,x)}=a+(1+x-c)
\frac{\psi'(a,c,x)}{\psi(a,c,x)}-x\left[\frac{\psi'(a,c,x)}{\psi(a,c,x)}\right]^2
=\left[\frac{x\psi'(a,c,x)}{\psi(a,c,x)}\right]'$$ which in view of
\eqref{quotrico} implies that
$$\frac{{}_{\psi}\Delta_c(x)}{\psi^2(a,c,x)}=
-\frac{a}{1+a-c}\int_0^{\infty}\frac{t\varphi_{a,c}(t)\dt}{(x+t)^2}<0$$
and
$$\left[\frac{{}_{\psi}\Delta_c(x)}{\psi^2(a,c,x)}\right]'=
\frac{a}{1+a-c}\int_0^{\infty}\frac{2t\varphi_{a,c}(t)\dt}{(x+t)^3}>0$$ for all
$a>0,$ $c<1$ and $x>0.$ Hence, the function $x\mapsto
{{}_{\psi}\Delta_{c}(x)}/{\psi^2(a,c,x)}$ maps $(0,\infty)$ into
$(-\infty,0)$ and it is strictly increasing. From this we obtain for
all $a>0,$ $c<1$ and $x>0$
$$\alpha_{c}:=\lim_{x\to 0}\frac{{}_{\psi}\Delta_{c}(x)}{\psi^2(a,c,x)}
<\frac{{}_{\psi}\Delta_{c}(x)}{\psi^2(a,c,x)}
<\lim_{x\to\infty}\frac{{}_{\psi}\Delta_{c}(x)}{\psi^2(a,c,x)}=:\beta_{c},$$
where $(1+a-c)\alpha_c=a/c$ and $\beta_c=0$ in view of the
asymptotic expansions \eqref{asymp1} and \eqref{asymp2}.

Now, let us focus on the inequality \eqref{turantrico4}. By using
the Kummer transformation \cite[p. 505]{abra}
\begin{equation}\label{kummertrans}\psi(a,c,x)=x^{1-c}\psi(1+a-c,2-c,x)\end{equation} the Tur\'an type
inequality \eqref{turantrico} becomes
$$\frac{1}{c}\psi^2(1+a-c,2-c,x)<\psi^2(1+a-c,2-c,x)-\psi(1+a-c,3-c,x)\psi(1+a-c,1-c,x)<0.$$
Replacing $a$ by $a+c-1$ we obtain
$$\frac{1}{c}\psi^2(a,2-c,x)<\psi^2(a,2-c,x)-\psi(a,3-c,x)\psi(a,1-c,x)<0.$$
The replacement of $c$ by $2-c$ gives \eqref{turantrico4}. The
sharpness of inequality \eqref{turantrico4} follows from the large
$x$ asymptotic expansion \eqref{asymp1} and from the expansion
$$\psi(a,c,x) \sim \frac{\Gamma(c-1)}{\Gamma(a)}x^{1-c},\ \ \mbox{as}\ \ x\to 0,$$ which is
valid for fixed $a$ and $c$ if $c>1.$
\end{proof}

\begin{remark}
{\em We note that in view of \eqref{whittaker} the Tur\'an type
inequality \eqref{turantrico3} for $x>0$ in terms of Whittaker
functions $W_{\kappa,\mu}$ reads as follows
$$-\frac{\mu-\kappa+1/2}{(1+2\mu)(\mu+\kappa+1/2)}W_{\kappa,\mu}^2(x)
<W_{\kappa,\mu}^2(x)-W_{\kappa-\frac{1}{2},\mu-\frac{1}{2}}(x)W_{\kappa+\frac{1}{2},\mu+\frac{1}{2}}(x)<0,$$
where the left-hand side holds for $0>\mu+1/2>\kappa,$ while the
right-hand side is valid for $\mu+1/2>\kappa.$ Similarly, the
inequality \eqref{turantrico4} can be rewritten as
$$\frac{1}{1-2\mu}W_{\kappa,\mu}^2(x)<W_{\kappa,\mu}^2(x)-
W_{\kappa-\frac{1}{2},\mu-\frac{1}{2}}(x)W_{\kappa+\frac{1}{2},\mu+\frac{1}{2}}(x)<0,$$
where the left-hand side holds for $1<\mu+1/2<1-\kappa,$ while the
right-hand side is valid for $1/2<\mu+1/2<-\kappa.$ Moreover, it
should be mentioned here that by using the H\"older-Rogers
inequality as in Remark \ref{remholder}, the right-hand side of the
above inequalities can be generalized in the following way: the
two-variable function
$$(\kappa,\mu)\mapsto \Gamma\left(\mu-\kappa+\frac{1}{2}\right)W_{\kappa,\mu}(x)=
\exp\left(-\frac{x}{2}\right)x^{\mu+\frac{1}{2}}
\int_0^{\infty}e^{-xt}t^{\mu-\kappa-\frac{1}{2}}(1+t)^{\mu+\kappa-\frac{1}{2}}\dt$$
is logarithmically convex for $\mu+1/2>\kappa$ and fixed
$x\in\mathbb{R}.$ Finally, observe that the above strict
logarithmic convexity property implies also the inequality
$$W_{\kappa,\mu}^2(x)-
W_{\kappa+\frac{1}{2},\mu-\frac{1}{2}}(x)W_{\kappa-\frac{1}{2},\mu+\frac{1}{2}}(x)
<\frac{1}{\mu-\kappa+1/2}W_{\kappa,\mu}^2(x),$$
however, this is weaker than the right-hand side of
\eqref{turanwhitta}.}
\end{remark}

\begin{remark}
{\em Observe that the Kummer transformation \eqref{kummertrans} is
also useful to prove the right-hand side of \eqref{turantrico}. More
precisely, from \eqref{kummertrans} we obtain
\begin{equation}\label{tricorepre}\Gamma(1+a-c)\psi(a,c,x)=x^{1-c}\Gamma(1+a-c)\psi(1+a-c,2-c,x)=
\int_0^{\infty}x^{1-c}e^{-xt}t^{a-c}(1+t)^{-a}\dt\end{equation} and
by using the H\"older-Rogers inequality, as in Remark
\ref{remholder}, we conclude that the two-variable function
$(a,c)\mapsto \Gamma(1+a-c)\psi(a,c,x)$ is strictly logarithmically
convex and consequently the right-hand side of the Tur\'an type
inequality \eqref{turantrico} is valid for all $c<a+1$ and $x>0.$}
\end{remark}

\section{Tur\'an determinants of Tricomi confluent hypergeometric functions}
\setcounter{equation}{0}

In this section we discuss the connection between the present paper
and \cite{il}. For this let us consider the determinants
$${}_1\mathrm{Det}_n(x):=\left|\begin{array}{cccc}
g(a,c,x)&g(a+1,c,x)&\cdots&g(a+n,c,x)\\
g(a+1,c,x)&g(a+2,c,x)&\cdots&g(a+n+1,c,x)\\
\vdots&\vdots& &\vdots\\
g(a+n,c,x)&g(a+n+1,c,x)&\cdots&g(a+2n,c,x)\end{array}\right|$$ and
$${}_2\mathrm{Det}_n(x):=\left|\begin{array}{cccc}
g(a,c,x)&g(a,c+1,x)&\cdots&g(a,c+n,x)\\
g(a,c+1,x)&g(a,c+2,x)&\cdots&g(a,c+n+1,x)\\
\vdots&\vdots& &\vdots\\
g(a,c+n,x)&g(a,c+n+1,x)&\cdots&g(a,c+2n,x)\end{array}\right|,$$
where $g(a,c,x):=\Gamma(a)\psi(a,c,x).$ In \cite{il} Ismail and
Laforgia stated that for all $a>0,$ $c\in\mathbb{R}$ and
$n\in\{0,1,\dots\}$ the determinants ${}_1\mathrm{Det}_n(x)$ and
${}_2\mathrm{Det}_n(x)$ are completely monotonic on $(0,\infty)$
with respect to $x,$ that is, we have
\begin{equation}\label{deter1}(-1)^m{}_1\mathrm{Det}_n^{(m)}(x)\geq0\end{equation}
and
\begin{equation}\label{deter11}(-1)^m{}_2\mathrm{Det}_n^{(m)}(x)\geq0\end{equation}
for all $a,x>0,$ $c\in\mathbb{R}$ and $n,m\in\{0,1,\dots\}.$ Observe
that if we choose in \eqref{deter1} the values $m=0,$ $n=1$ and
instead of $a$ we write $a-1,$ then we obtain the weak Tur\'an type
inequality \eqref{eqrem1}. Similarly, if we take in \eqref{deter11}
the values $m=0,$ $n=1$ and we write $c-1$ instead of $c,$ then we
get the right-hand side of the inequality \eqref{turantrico3}.

We now present a general result which is in the same spirit as
\cite[Remark 2.9]{il} and which generalizes the above mentioned
results from \cite{il}. For this consider
$\alpha,\beta\in\mathbb{R}$ such that $\alpha<\beta,$ and let
$\{f_n\}_{n\geq0}$ be a sequence of functions, defined by
\begin{equation}\label{type}f_n(x): = \int_{\alpha}^{\beta} [\f(t,x)]^n d\mu(t,x),\end{equation} where
$\f,\mu:[\alpha,\beta]\times\mathbb{R}\to\mathbb{R}.$ Consider also the determinant
$$\mathrm{Det}_n(x):=\left| \begin{array}{cccc}
f_0(x) & f_{1}(x) & \cdots & f_{n}(x)\\
f_{1}(x) & f_{2}(x) & \cdots & f_{n+1}(x)\\
\vdots & \vdots &  & \vdots\\
f_{n}(x) & f_{n+1}(x) & \cdots & f_{2n}(x)\\
\end{array} \right|.$$
Then the following result is valid.

\begin{theorem}\label{th5}
We have the following representation
\begin{equation}\label{deterrepre}
\mathrm{Det}_n(x) = \frac{1}{(n+1)!} \int_{[\alpha,\beta]^{n+1}}
\prod_{0 \le j < k \le n} [\f(t_j,x) - \f(t_k,x)]^2 \prod_{j=0}^n
d\mu(t_j,x).
\end{equation}
\end{theorem}

\begin{proof}[\bf Proof]
Observe that
\begin{align*}
\mathrm{Det}_n(x)&= \int_{[\alpha,\beta]^{n+1}} \left|
\begin{array}{cccc}
1 & \f(t_0,x) & \cdots & [\f(t_0,x)]^n \\
\f(t_1,x)  &[ \f(t_1,x)]^2 & \cdots & [\f(t_1,x)]^{n+1}\\
\vdots & \vdots & {} & \vdots\\
{} [\f(t_n,x)]^n & [\f(t_n,x)]^{n+1} & \cdots & [\f(t_n,x)]^{2n}  \\
\end{array} \right|
\prod_{j=0}^n  d\mu(t_j,x) \\
&=
\int_{[\alpha,\beta]^{n+1}}
\left| \begin{array}{cccc}
1 & \f(t_0,x) & \cdots & [\f(t_0,x)]^n \\
1 & \f(t_1,x) & \cdots & [\f(t_1,x)]^{n}\\
\vdots & \vdots & {} & \vdots\\
1 & \f(t_n,x) & \cdots & [\f(t_n,x)]^{n}  \\
\end{array} \right|
\prod_{j=0}^n  [\f(t_j,x)]^jd\mu(t_j,x)\\
&= \textup{sign}(\sigma)\int_{[\alpha,\beta]^{n+1}} \left|
\begin{array}{cccc}
1 & \f(t_{\sigma(0)},x) & \cdots & [\f(t_{\sigma(0)},x)]^n \\
1 & \f(t_{\sigma(1)},x) & \cdots & [\f(t_{\sigma(1)},x)]^{n}\\
\vdots & \vdots & {} & \vdots\\
{} 1 & \f(t_{\sigma(n)},x) & \cdots & [\f(t_{\sigma(n)},x)]^{n}  \\
\end{array} \right| \prod_{j=0}^n  [\f(t_j,x) ]^{\sigma(j)} d\mu(t_j,x),
\end{align*}
where $\sigma$ is a permutation on $\{0,1, \dots, n\}.$ The
determinant in the last expression is a Vandermonde determinant
which can be evaluated as a product. Thus, if we add over all possible
$\sigma$ and divide by $(n+1)!,$ then we can see that $\mathrm{Det}_n(x)$ is given by the right-hand side of \eqref{deterrepre} because
$$\sum_{\sigma} \textup{sign}(\sigma) \prod_{j=0}^n [\f(t_j,x)]^{\sigma(j)} =
\prod_{0 \le j < k \le n} [\f(t_j,x) - \f(t_k,x)].$$
\end{proof}

Note that the proof above is similar to Heine's classical proof of
his integral representation, see for example  \cite{Sze} and \cite{Ismbook}.
 We also note that properties of
Tur\'an determinants of which entries are orthogonal polynomial
families and special functions have been discussed also in the
papers \cite{ismail,il,karlin}. See also the book \cite{Ismbook} for
more details.

Now, observe that by using \eqref{intlap} and \eqref{deterrepre} we
easily obtain
\begin{align*}{}_1\mathrm{Det}_n(x)=\frac{1}{(n+1)!}&\int_{[0,\infty)^{n+1}}\exp\left(-x\sum_{j=0}^nt_j\right)
\\&\times\prod_{0 \le j < k \le
n}\left(\frac{t_j}{t_j+1}-\frac{t_k}{t_k+1}\right)^2\prod_{j=0}^nt_j^{a-1}(1+t_j)^{c-a-1}\dt_j\end{align*}
and
\begin{align*}{}_2\mathrm{Det}_n(x)=\frac{1}{(n+1)!}&\int_{[0,\infty)^{n+1}}\exp\left(-x\sum_{j=0}^nt_j\right)
\\&\times\prod_{0 \le j < k \le
n}\left({t_j}-{t_k}\right)^2\prod_{j=0}^nt_j^{a-1}(1+t_j)^{c-a-1}\dt_j\end{align*}
which clearly imply \eqref{deter1} and \eqref{deter11}.

Furthermore, if we consider the determinants
$${}_3\mathrm{Det}_n(x):=\left|\begin{array}{cccc}
h(a,c,x)&h(a+1,c,x)&\dots&h(a+n,c,x)\\
h(a+1,c,x)&h(a+2,c,x)&\dots&h(a+n+1,c,x)\\
\vdots&\vdots& &\vdots\\
h(a+n,c,x)&h(a+n+1,c,x)&\dots&h(a+2n,c,x)\end{array}\right|$$ and
$${}_4\mathrm{Det}_n(x):=\left|\begin{array}{cccc}
h(a,c,x)&h(a,c+1,x)&\dots&h(a,c+n,x)\\
h(a,c+1,x)&h(a,c+2,x)&\dots&h(a,c+n+1,x)\\
\vdots&\vdots& &\vdots\\
h(a,c+n,x)&h(a,c+n+1,x)&\dots&h(a,c+2n,x)\end{array}\right|,$$ where
$h(a,c,x):=\Gamma(1+a-c)\psi(a,c,x),$ then by using
\eqref{tricorepre} and \eqref{deterrepre} we obtain

\begin{align*}{}_3\mathrm{Det}_n(x)=\frac{1}{(n+1)!}&\int_{[0,\infty)^{n+1}}\exp\left(-x\sum_{j=0}^nt_j\right)
\\&\times\prod_{0 \le j < k \le
n}\left(\frac{t_j}{t_j+1}-\frac{t_k}{t_k+1}\right)^2\prod_{j=0}^nx^{1-c}t_j^{a-c}(1+t_j)^{-a}\dt_j\end{align*}
and
\begin{align*}{}_4\mathrm{Det}_n(x)=\frac{1}{(n+1)!}&\int_{[0,\infty)^{n+1}}\exp\left(-x\sum_{j=0}^nt_j\right)
\\&\times\prod_{0 \le j < k \le
n}\left(\frac{1}{xt_j}-\frac{1}{xt_k}\right)^2\prod_{j=0}^nx^{1-c}t_j^{a-c}(1+t_j)^{-a}\dt_j.\end{align*}

Similarly, if we consider the determinants
$${}_5\mathrm{Det}_n(x):=\left|\begin{array}{cccc}
g(a,c,x)&g(a+1,c+1,x)&\cdots&g(a+n,c+n,x)\\
g(a+1,c+1,x)&g(a+2,c+2,x)&\cdots&g(a+n+1,c+n+1,x)\\
\vdots&\vdots& &\vdots\\
g(a+n,c+n,x)&g(a+n+1,c+n+1,x)&\cdots&g(a+2n,c+2n,x)\end{array}\right|$$
and
$${}_6\mathrm{Det}_n(x):=\left|\begin{array}{cccc}
h(a,c,x)&h(a+1,c+1,x)&\dots&h(a+n,c+n,x)\\
h(a+1,c+1,x)&h(a+2,c+2,x)&\dots&h(a+n+1,c+n+1,x)\\
\vdots&\vdots& &\vdots\\
h(a+n,c+n,x)&h(a+n+1,c+n+1,x)&\dots&h(a+2n,c+2n,x)\end{array}\right|,$$
then by using \eqref{intlap}, \eqref{tricorepre} and
\eqref{deterrepre} we get
\begin{align*}{}_5\mathrm{Det}_n(x)=\frac{1}{(n+1)!}&\int_{[0,\infty)^{n+1}}\exp\left(-x\sum_{j=0}^nt_j\right)
\\&\times\prod_{0 \le j < k \le
n}\left({t_j}-{t_k}\right)^2\prod_{j=0}^nt_j^{a-1}(1+t_j)^{c-a-1}\dt_j\end{align*}
and
\begin{align*}{}_6\mathrm{Det}_n(x)=\frac{1}{(n+1)!}&\int_{[0,\infty)^{n+1}}\exp\left(-x\sum_{j=0}^nt_j\right)
\\&\times\prod_{0 \le j < k \le
n}\left(\frac{1}{x(t_j+1)}-\frac{1}{x(t_k+1)}\right)^2\prod_{j=0}^nx^{1-c}t_j^{a-c}(1+t_j)^{-a}\dt_j.\end{align*}

Finally, if we consider the determinants
$${}_7\mathrm{Det}_n(x):=\left|\begin{array}{cccc}
g(a,c,x)&g(a+1,c+2,x)&\cdots&g(a+n,c+2n,x)\\
g(a+1,c+2,x)&g(a+2,c+4,x)&\cdots&g(a+n+1,c+2n+2,x)\\
\vdots&\vdots& &\vdots\\
g(a+n,c+2n,x)&g(a+n+1,c+2n+2,x)&\cdots&g(a+2n,c+4n,x)\end{array}\right|$$
and
$${}_8\mathrm{Det}_n(x):=\left|\begin{array}{cccc}
h(a,c,x)&h(a+1,c+2,x)&\dots&h(a+n,c+2n,x)\\
h(a+1,c+2,x)&h(a+2,c+4,x)&\dots&h(a+n+1,c+2n+2,x)\\
\vdots&\vdots& &\vdots\\
h(a+n,c+2n,x)&h(a+n+1,c+2n+2,x)&\dots&h(a+2n,c+4n,x)\end{array}\right|,$$
then by using again \eqref{intlap}, \eqref{tricorepre} and
\eqref{deterrepre} we obtain
\begin{align*}{}_7\mathrm{Det}_n(x)=\frac{1}{(n+1)!}&\int_{[0,\infty)^{n+1}}\exp\left(-x\sum_{j=0}^nt_j\right)
\\&\times\prod_{0 \le j < k \le
n}\left({t_j}-{t_k}\right)^2\left({t_j}+{t_k}+1\right)^2\prod_{j=0}^nt_j^{a-1}(1+t_j)^{c-a-1}\dt_j\end{align*}
and
\begin{align*}{}_8\mathrm{Det}_n(x)=\frac{1}{(n+1)!}&\int_{[0,\infty)^{n+1}}\exp\left(-x\sum_{j=0}^nt_j\right)
\\&\times\prod_{0 \le j < k \le
n}\left(\frac{1}{xt_j(t_j+1)}-\frac{1}{xt_k(t_k+1)}\right)^2\prod_{j=0}^nx^{1-c}t_j^{a-c}(1+t_j)^{-a}\dt_j.\end{align*}

Now, taking into account the well-known fact that the product of
completely monotonic functions is also completely monotonic, the
above integral representations imply the following result, which
complements \cite[Theorem 2.8]{il}.

\begin{theorem}\label{th6}
Let $n\in\{0,1,\dots\}.$ If $a+1>c>1,$ then the determinants
${}_3\mathrm{Det}_n(x),$ ${}_4\mathrm{Det}_n(x),$
${}_6\mathrm{Det}_n(x)$ and ${}_8\mathrm{Det}_n(x)$ are completely
monotonic on $(0,\infty)$ with respect to $x.$ Moreover, the
determinants ${}_5\mathrm{Det}_n(x)$ and ${}_7\mathrm{Det}_n(x)$ are
also completely monotonic on $(0,\infty)$ with respect to $x$ for
all $a>0$ and $c\in\mathbb{R}.$
\end{theorem}

\begin{remark}\label{rem7}
{\em We note that the above results complement the main results of
section 2. To see this in what follows we will discuss the
particular cases of the above results. Since for admissible values
of $a$ and $c$ (given in Theorem \ref{th6}) and for
$n\in\{0,1,\dots\}$ the functions $x\mapsto{}_i\mathrm{Det}_n(x),$
where $i\in\{3,\dots,8\},$ are completely monotonic on $(0,\infty)$,
for those values of $a$ and $c$ and for all $n,m\in\{0,1,\dots\}$
and $x>0$ we have
\begin{equation}\label{deter2}(-1)^m{}_i\mathrm{Det}_n^{(m)}(x)>0.\end{equation}
We note that if $i=3$ and we choose in \eqref{deter2} the values
$m=0,$ $n=1$ and instead of $a$ we write $a-1,$ then we reobtain for
$a>c>1$ and $x>0$ the left-hand side of the sharp Tur\'an type
inequality \eqref{turantricom}. Similarly, if $i=4$ and we take in
\eqref{deter2} the values $m=0,$ $n=1$ and we write $c-1$ instead of
$c,$ then for all $a+1>c>2$ and $x>0$ we get the Tur\'an type
inequality
$$\psi^2(a,c,x)-\psi(a,c-1,x)\psi(a,c+1,x)<\frac{1}{a-c+1}\psi^2(a,c,x),$$
however, this is weaker than the right-hand side of the sharp
inequality \eqref{turantrico3} or \eqref{turantrico4}. Moreover, by
using the inequality \eqref{deter2} for $i=6,$ $m=0,$ $n=1$ and then
changing $a$ with $a-1$ and $c$ with $c-1,$ we obtain that the
right-hand of the sharp Tur\'an type inequality \eqref{turantrico}
is also valid for $a+1>c>2$ and $x>0.$ Now, let $i=8.$ Then by using
again the inequality \eqref{deter2} for $m=0,$ $n=1$ and then
changing $a$ with $a-1$ and $c$ with $c-2,$ we obtain the Tur\'an
type inequality
$$\psi^2(a,c,x)-\psi(a-1,c-2,x)\psi(a+1,c+2,x)<\frac{1}{a-c+1}\psi^2(a,c,x),$$
which is valid for all $a+1>c>3$ and $x>0.$ This resembles to
\eqref{turanhankel}. Analogously, if $i=5$ and we take in
\eqref{deter2} the values $m=0,$ $n=1$ and we write $a-1$ instead of
$a,$ $c-1$ instead of $c,$ then for all $a>1,$ $c\in\mathbb{R}$ and
$x>0$ we get the Tur\'an type inequality
$$\psi^2(a,c,x)-\psi(a,c-1,x)\psi(a,c+1,x)<\frac{1}{a}\psi^2(a,c,x),$$
however, this is weaker than the right-hand side of the sharp
inequality \eqref{turantrico3} or \eqref{turantrico4}. Finally, let
$i=7.$ By using again the inequality \eqref{deter2} for $m=0,$ $n=1$
and then changing $a$ with $a-1$ and $c$ with $c-2,$ we obtain the
Tur\'an type inequality
\begin{equation}\label{turanweak}\psi^2(a,c,x)-\psi(a-1,c-2,x)\psi(a+1,c+2,x)<\frac{1}{a}\psi^2(a,c,x),\end{equation}
which is valid for all $a>1,$ $c\in\mathbb{R}$ and $x>0.$ This also
resembles to \eqref{turanhankel}. Moreover, if we take in
\eqref{turanweak} instead of $a$ the value $a+1/2$ and instead of
$c$ the value $2a+1,$ then the above inequality becomes
$$\Delta_a(x)<\frac{1}{a+1/2}\psi^2\left(a+\frac{1}{2},2a+1,x\right),$$
which is valid for all $a>1/2$ and $x>0.$ This complements
the left-hand side of \eqref{turanhankel}, however, it is weaker
than the right-hand side of \eqref{turanhankel}.}
\end{remark}

\begin{remark}
{\em It should be mentioned here that similar results to those
mentioned in Theorem \ref{th6} are also valid for the Kummer
confluent hypergeometric $\Phi(a,c,\cdot).$ Namely, if we consider
the determinant
$${}_9\mathrm{Det}_n(x):=\left|\begin{array}{cccc}
u(a,c,x)&u(a,c+1,x)&\cdots&u(a,c+n,x)\\
u(a,c+1,x)&u(a,c+2,x)&\cdots&u(a,c+n+1,x)\\
\vdots&\vdots& &\vdots\\
u(a,c+n,x)&u(a,c+n+1,x)&\cdots&u(a,c+2n,x)\end{array}\right|,$$
where
$$u(a,c,x):=\frac{\Gamma(c-a)}{\Gamma(c)}\Phi(a,c,x)=\frac{1}{\Gamma(a)}\int_0^1e^{xt}t^{a-1}(1-t)^{c-a-1}\dt,$$
then applying \eqref{deterrepre} we get
\begin{align*}{}_9\mathrm{Det}_n(x)=\frac{1}{(n+1)!}\frac{1}{\Gamma^{n+1}(a)}&\int_{[0,1]^{n+1}}\exp\left(x\sum_{j=0}^nt_j\right)
\\&\times\prod_{0 \le j < k \le
n}\left({t_k}-{t_j}\right)^2\prod_{j=0}^nt_j^{a-1}(1-t_j)^{c-a-1}\dt_j.\end{align*}
This in turn implies a known result of Ismail and Laforgia
\cite[Theorem 2.7]{il}. Namely, for $c>a>0$ and $n\in\{0,1,\dots\}$
the determinant ${}_9\mathrm{Det}_n(x)$ as a function of $x$ is
absolutely monotonic on $(0,\infty),$ i.e. for all
$n,m\in\{0,1,\dots\},$ $c>a>0$ and $x>0$ we have
$${}_9\mathrm{Det}_n^{(m)}(x)\geq0.$$ Moreover, by using again \eqref{deterrepre} the determinant
$${}_{10}\mathrm{Det}_n(x):=\left|\begin{array}{cccc}
v(a,c,x)&v(a+1,c,x)&\dots&v(a+n,c,x)\\
v(a+1,c,x)&v(a+2,c,x)&\dots&v(a+n+1,c,x)\\
\vdots&\vdots& &\vdots\\
v(a+n,c,x)&v(a+n+1,c,x)&\dots&v(a+2n,c,x)\end{array}\right|,$$ where
$v(a,c,x):=\Gamma(a)\Gamma(c-a)\Phi(a,c,x),$ can be rewritten as
\begin{align*}{}_{10}\mathrm{Det}_n(x)=\frac{\Gamma^{n+1}(c)}{(n+1)!}&\int_{[0,1]^{n+1}}\exp\left(x\sum_{j=0}^nt_j\right)
\\&\times\prod_{0 \le j < k \le
n}\left(\frac{t_j}{1-t_j}-\frac{t_k}{1-t_k}\right)^2\prod_{j=0}^nt_j^{a-1}(1-t_j)^{c-a-1}\dt_j.\end{align*}
Consequently for $c>a>0$ and $n\in\{0,1,\dots\}$ the determinant
${}_{10}\mathrm{Det}_n(x)$ as a function of $x$ is also absolutely
monotonic on $(0,\infty),$ and thus for all $n,m\in\{0,1,\dots\},$
$c>a>0$ and $x>0$ we have
$${}_{10}\mathrm{Det}_n^{(m)}(x)\geq0.$$ Now, if we take $m=0$ and $n=1,$ and we change $a$ to $a-1$ we obtain the following Tur\'an type inequality
$$\frac{\Phi(a-1,c,x)\Phi(a+1,c,x)}{\Phi^2(a,c,x)}\geq\frac{(a-1)(c-a-1)}{a(c-a)},$$
where $c>a+1>2$ and $x>0.$ This inequality complements the result of Barnard et
al. \cite[Corollary 2]{barnard}, is similar to the result of Karp
\cite[Corollary 3]{karp0} and it is weaker than the left-hand side of \cite[Eq. (10)]{karp}.

Finally, we also note that the determinant
$${}_{11}\mathrm{Det}_n(x):=\left|\begin{array}{cccc}
w(a,c,x)&w(a+1,c+1,x)&\cdots&w(a+n,c+n,x)\\
w(a+1,c+1,x)&w(a+2,c+2,x)&\cdots&w(a+n+1,c+n+1,x)\\
\vdots&\vdots& &\vdots\\
w(a+n,c+n,x)&w(a+n+1,c+n+1,x)&\cdots&w(a+2n,c+2n,x)\end{array}\right|,$$
where
$$w(a,c,x):=\frac{\Gamma(a)\Gamma(c-a)}{\Gamma(a)}\Phi(a,c,x)=\int_0^1e^{xt}t^{a-1}(1-t)^{c-a-1}\dt,$$
can be rewritten as follows
\begin{align*}{}_{11}\mathrm{Det}_n(x)=\frac{1}{(n+1)!}&\int_{[0,1]^{n+1}}\exp\left(x\sum_{j=0}^nt_j\right)
\\&\times\prod_{0 \le j < k \le
n}\left({t_j}-{t_k}\right)^2\prod_{j=0}^nt_j^{a-1}(1-t_j)^{c-a-1}\dt_j.\end{align*}
This in turn implies that for $c>a>0$ and $n\in\{0,1,\dots\}$ the
determinant ${}_{11}\mathrm{Det}_n(x)$ as a function of $x$ is also
absolutely monotonic on $(0,\infty),$ and thus for all
$n,m\in\{0,1,\dots\},$ $c>a>0$ and $x>0$ we have
$${}_{11}\mathrm{Det}_n^{(m)}(x)\geq0.$$ Now, if we take in this inequality $m=0$ and $n=1,$ and
we change $a$ to $a-1,$ $c$ to $c-1$ we obtain the following Tur\'an
type inequality
$$\frac{\Phi(a-1,c-1,x)\Phi(a+1,c+1,x)}{\Phi^2(a,c,x)}\geq\frac{(a-1)c}{(c-1)a},$$
where $c>a>1$ and $x>0.$ This is a particular case of the left-hand
side inequality of \cite[Corollary 3]{karp0} and is the counterpart
of \cite[Theorem 1]{karp0}. See also \cite[Theorem 2]{baricz2} for a
similar Tur\'an type inequality.}
\end{remark}

We end the paper with the following remark.

\begin{remark}
{\em We note that many other special functions have representations
of the type \eqref{type}. For example, the modified Bessel functions
of the first and second kind $I_a$ and $K_a$ are like this. Thus, by
using the idea used in this section we can explore this further to
get positivity of many determinants which entries are special
functions. Also we may be able to prove that the determinants are
completely (or absolutely) monotonic if the integral representation
\eqref{type} involves completely (absolutely) monotonic kernels. See
\cite{bapo} for similar results on the so-called Kr\"atzel function.

Now, we would like to complement the results of Theorem 2.3 and 2.5
from \cite{il}. For this we consider first for $a>-1/2$ and $x>0$
the integral representation (see \cite[p. 376]{abra} or \cite[p.
172]{Watson})
$$ K_a(x) =
\frac{\sqrt{\pi}\left(\frac{x}{2}\right)^a}{\Gamma\left(a+\frac{1}{2}\right)}
\int_1^\infty e^{-xt}\, (t^2-1)^{a-\frac{1}{2}} \dt$$ and apply
Theorem \ref{th5} for
$$f_n(x) = \frac{\Gamma\left(a+n+\frac{1}{2}\right)}{\sqrt{\pi}\left(\frac{x}{2}\right)^{a+n}} e^xK_{a+n}(x)$$
to get
\begin{align*}
{}_{12}\mathrm{Det}_n(x)&=\left|\begin{array}{cccc}u_a(x)&u_{a+1}(x)&\cdots&u_{a+n}(x)\\
u_{a+1}(x)&u_{a+2}(x)&\cdots&u_{a+n+1}(x)\\
\vdots&\vdots& &\vdots\\
u_{a+n}(x)&u_{a+n+1}(x)&\cdots&u_{a+2n}(x)\end{array}\right|\\
&=\frac{1}{(n+1)!}\int_{[1,\infty)^{n+1}} \exp\left(-x\sum_{j=0}^n
(t_j-1)\right)\prod_{0 \le j < k \le n} \left(t_j^2 - t_k^2\right)^2
 \prod_{j=0}^n (t_j^2-1)^{a -\frac{1}{2}} dt_j,
\end{align*}
where
$$u_a(x):=\frac{\Gamma\left(a+\frac{1}{2}\right)}{\sqrt{\pi}\left(\frac{x}{2}\right)^{a}}
e^xK_{a}(x).$$ Clearly, the determinant ${}_{12}\mathrm{Det}_n(x)$ is
completely monotonic on $(0,\infty)$ for each $n\in\{0,1,\dots\}$
and $a>-1/2.$ Now, for $a>-1/2$ consider the integral representation
(see \cite[p. 376]{abra} or \cite[p. 79]{Watson})
$$I_a(x) = \frac{\left(\frac{x}{2}\right)^a}{\sqrt{\pi}\Gamma\left(a +\frac{1}{2}\right)} \int_{-1}^1 e^{\pm xt}(1-t^2)^{a-\frac{1}{2}}\dt.$$
Here we have two choices for $f_n$.  They are
$$f_n(x) = \frac{\sqrt{\pi}\Gamma\left(a+n+\frac{1}{2}\right)}{\left(\frac{x}{2}\right)^{a+n}}e^x I_{a+n}(x)$$
or $$f_n(x) =
\frac{\sqrt{\pi}\Gamma\left(a+n+\frac{1}{2}\right)}{\left(\frac{x}{2}\right)^{a+n}}e^{-x}
I_{a+n}(x).$$ One is led to an absolutely monotonic determinant and
one to completely monotonic determinant. More precisely, as a
function of $x$ and for $a>-1/2$ and $n\in\{0,1,\dots\}$ the
determinant
\begin{align*}
{}_{13}\mathrm{Det}_n(x)&=\left|\begin{array}{cccc}v_a(x)&v_{a+1}(x)&\cdots&v_{a+n}(x)\\
v_{a+1}(x)&v_{a+2}(x)&\cdots&v_{a+n+1}(x)\\
\vdots&\vdots& &\vdots\\
v_{a+n}(x)&v_{a+n+1}(x)&\cdots&v_{a+2n}(x)\end{array}\right|\\
&=\frac{1}{(n+1)!}\int_{[-1,1]^{n+1}} \exp\left(x\sum_{j=0}^n
(1-t_j)\right)\prod_{0 \le j < k \le n} \left(t_j^2 - t_k^2\right)^2
 \prod_{j=0}^n (t_j^2-1)^{a -\frac{1}{2}} dt_j,
\end{align*}
where
$$v_a(x):=\frac{\sqrt{\pi}\Gamma\left(a+\frac{1}{2}\right)}{\left(\frac{x}{2}\right)^{a}}
e^xI_{a}(x)=\int_{-1}^1 e^{(1-t)x}(1-t^2)^{a-\frac{1}{2}}\dt,$$ is
absolutely monotonic on $(0,\infty),$ while the determinant
\begin{align*}
{}_{14}\mathrm{Det}_n(x)&=\left|\begin{array}{cccc}w_a(x)&w_{a+1}(x)&\cdots&w_{a+n}(x)\\
w_{a+1}(x)&w_{a+2}(x)&\cdots&w_{a+n+1}(x)\\
\vdots&\vdots& &\vdots\\
w_{a+n}(x)&w_{a+n+1}(x)&\cdots&w_{a+2n}(x)\end{array}\right|\\
&=\frac{1}{(n+1)!}\int_{[-1,1]^{n+1}} \exp\left(-x\sum_{j=0}^n
(1-t_j)\right)\prod_{0 \le j < k \le n} \left(t_j^2 - t_k^2\right)^2
 \prod_{j=0}^n (t_j^2-1)^{a -\frac{1}{2}} dt_j,
\end{align*}
where
$$w_a(x):=\frac{\sqrt{\pi}\Gamma\left(a+\frac{1}{2}\right)}{\left(\frac{x}{2}\right)^{a}}
e^{-x}I_{a}(x)=\int_{-1}^1 e^{-(1-t)x}(1-t^2)^{a-\frac{1}{2}}\dt,$$
is completely monotonic on $(0,\infty).$ We mention here that for
$n=1$ the above results lead to weak Tur\'an type inequalities for
modified Bessel functions of the first and second kinds, which were
mentioned already in Remarks 2.4 and 2.6 in \cite{il}.

Finally, it is worth to mention that the above method can be used
also to prove absolute and complete monotonic properties of
determinants whose entries are probability density functions. For
example, for the probability density function of the non-central chi
distribution we can prove such results. More precisely, if we
consider the probability density function
$\chi_{a,\tau}:(0,\infty)\to(0,\infty)$ of the non-central chi
distribution (see \cite{johnson}) with shape parameter $a>0$ and
non-centrality parameter $\tau>0,$ defined by
$$\chi_{a,\tau}(x):=\tau e^{-\frac{x^2+\tau^2}{2}}\left(\frac{x}{\tau}\right)^{\frac{a}{2}}I_{\frac{a}{2}-1}(\tau x),$$
then it is not difficult to see that as a function of $x$ and for
all $a>1/2,$ $\tau>0$ and $n\in\{0,1,\dots\}$ the determinant
\begin{align*}
{}_{15}\mathrm{Det}_n(x)&=\left|\begin{array}{cccc}r_{2a}(x)&r_{2a+1}(x)&\cdots&r_{2a+n}(x)\\
r_{2a+1}(x)&r_{2a+2}(x)&\cdots&r_{2a+n+1}(x)\\
\vdots&\vdots& &\vdots\\
r_{2a+n}(x)&r_{2a+n+1}(x)&\cdots&r_{2a+2n}(x)\end{array}\right|\\
&=\frac{1}{(n+1)!}\int_{[-1,1]^{n+1}} \exp\left(\tau x\sum_{j=0}^n
(1-t_j)\right)\prod_{0 \le j < k \le n} \left(t_j^2 - t_k^2\right)^2
 \prod_{j=0}^n (t_j^2-1)^{a -\frac{3}{2}} dt_j,
\end{align*}
where
$$r_a(x):=\frac{\sqrt{\pi}2^{\frac{a}{2}-1}\Gamma\left(\frac{a-1}{2}\right)}{x^{a-1}}
e^{\frac{(x+\tau)^2}{2}}\chi_{a,\tau}(x)=\int_{-1}^1 e^{(1-t)\tau
x}(1-t^2)^{\frac{a-3}{2}}\dt,$$ is absolutely monotonic on
$(0,\infty),$ while the determinant
\begin{align*}
{}_{16}\mathrm{Det}_n(x)&=\left|\begin{array}{cccc}s_{2a}(x)&s_{2a+1}(x)&\cdots&s_{2a+n}(x)\\
s_{2a+1}(x)&s_{2a+2}(x)&\cdots&s_{2a+n+1}(x)\\
\vdots&\vdots& &\vdots\\
s_{2a+n}(x)&s_{2a+n+1}(x)&\cdots&s_{2a+2n}(x)\end{array}\right|\\
&=\frac{1}{(n+1)!}\int_{[-1,1]^{n+1}} \exp\left(-\tau x\sum_{j=0}^n
(1-t_j)\right)\prod_{0 \le j < k \le n} \left(t_j^2 - t_k^2\right)^2
 \prod_{j=0}^n (t_j^2-1)^{a -\frac{3}{2}} dt_j,
\end{align*}
where
$$s_a(x):=\frac{\sqrt{\pi}2^{\frac{a}{2}-1}\Gamma\left(\frac{a-1}{2}\right)}{x^{a-1}}
e^{\frac{(x-\tau)^2}{2}}\chi_{a,\tau}(x)=\int_{-1}^1 e^{-(1-t)\tau
x}(1-t^2)^{\frac{a-3}{2}}\dt,$$ is completely monotonic on
$(0,\infty).$ Note that the positivity of the above determinants for
$n=1$ yields the Tur\'an type inequality
$$\chi_{2a+1,\tau}^2(x)-\chi_{2a,\tau}(x)\chi_{2a+2,\tau}(x)<
\left[1-\frac{\Gamma^2(a)}{\Gamma\left(a-\frac{1}{2}\right)\Gamma(a+\frac{1}{2})}\right]\chi_{2a+1,\tau}^2(x),$$
where $a>1/2,$ $\tau>0$ and $x>0.$ This inequality completes
\cite[Theorem 2.3]{andras}, however it is weaker than the sharp
Tur\'an type inequality
$$0<\chi_{2a+1,\tau}^2(x)-\chi_{2a,\tau}(x)\chi_{2a+2,\tau}(x)<
\left[1-\frac{\Gamma^2\left(a+\frac{1}{2}\right)}{\Gamma(a)\Gamma(a+1)}\right]\chi_{2a+1,\tau}^2(x),$$
which holds for all $a,\tau,x>0$ and is a particular case of
\cite[Theorem 2.4]{bariczstudia}.}
\end{remark}

\subsection*{Acknowledgements}
The research of \'A. Baricz was supported in part by the J\'anos
Bolyai Research Scholarship of the Hungarian Academy of Sciences and
in part by the Romanian National Council for Scientific Research in
Education CNCSIS-UEFISCSU, project number PN-II-RU-PD\underline{
}388/2012, and was completed during his visit in September 2011 to City University of
Hong Kong. This author is grateful to the Department of Mathematics
of City University of Hong Kong for hospitality. The research of
M.E.H. Ismail was partially supported by the NPST Program of King Saud University,
Riyadh, project number 10-MAT 1293-02 and by the Research Grants Council of Hong
Kong under contract \# 101411. This work was carried out while M.E.H. Ismail was
affiliated with the Department of Mathematics, City University of Hong Kong,
Kowloon, Hong Kong. Both of the authors are grateful to the referees for extensive comments
and constructive criticisms that improved the presentation of the results.

\end{document}